\documentclass[12 pt]{amsart}
\usepackage{graphicx}
\usepackage{verbatim}
\usepackage{comment}
\usepackage{calc}
\usepackage{color}
\usepackage{soul}
\setul{}{2pt}
\usepackage{tkz-euclide}
\usetkzobj{all}
\usepackage{placeins} 
\usepackage[top=1.5in, bottom=1.5in, left=1in, right=1in]{geometry}
\usepackage{setspace}
\usepackage{qtree}
\usepackage{amsthm}
\usepackage{hyperref}

\newtheorem{theorem}{Theorem}[section]

\theoremstyle{instructions}
\newtheorem{instructions}[theorem]{Instructions}
\theoremstyle{definition}

\newtheorem{example}[theorem]{Example}

\theoremstyle{remark}

\theoremstyle{problem}
\newtheorem{problem}[theorem]{Problem}

\theoremstyle{properties}
\newtheorem{properties}[theorem]{Properties}

\begin{document}

\title{Features of a high school olympiad problem}


\author{Lawrence Smolinsky}
\address{Department of Mathematics\\
Louisiana State University\\
Baton Rouge, LA 70803, USA}
\email{smolinsk@math.lsu.edu}

\date{\today}

\begin{abstract}  This paper is a supplement to a talk for mathematics teachers given at the 2016 LSU Mathematics Contest for High School Students.  The paper covers more details and aspects than could be covered in the talk.

We start with an interesting problem from the 2009 Iberoamerican Math Olympiad concerning a particular sequence.  We include a solution to the problem, but also relate it to several areas of mathematics.  This problem demonstrates the countability of the rational numbers with a direct one-to-one correspondence.  The problem also shows the one-to-one correspondence of finite continued fractions and rational numbers.  The subsequence of odd indexed terms was constructed by Johannes Kepler and is discussed.  We also show that the indices have an extension to the 2-adic integers giving a one-to-one correspondence between the positive real numbers and the 2-adic integers.  Everything but the extension to the 2-adic integers is known to have appeared elsewhere.
\end{abstract}

\maketitle
\pagestyle{myheadings}
\markboth{L. Smolinsky}{Features of a high school olympiad problem}

\section{The Problem}

Let us start with problem five from the 2009 Iberoamerican Math Olympiad.

\begin{problem}
\label{thm:problem}
The sequence $\{ a_n\}_{n=1}^\infty$  satisfies $a_1 = 1$ and for $n\ge 1$,
\[a_{2n} = a_n + 1; \quad a_{2n+1} = \frac 1{a_{2n}}.\]
Prove that every positive rational number occurs in the sequence exactly once.
\end{problem}

Let's write out the beginning of the sequence:
\[1, 2, \frac 12, 3, \frac 13, \frac 32, \frac 23, 4, \frac 14, \frac 43, \frac 34, \frac 52, \frac 25, \frac 53, \frac 35, 5,  \frac 15, \frac 54, \frac 45, \frac 73,  \frac 37,  \frac 74, \frac 47, \frac 72, \frac 27, \frac 75, \frac 57, \frac 83,  \frac 38, \cdots\]

\section{Countable and uncountable.}  For a finite set we can determine how large it is by counting the number of elements.  One can view counting the elements in a finite set as putting them in one-to-one correspondence with a finite counting set.  For example,  $\{ \clubsuit, \diamondsuit, \spadesuit, \heartsuit \} \leftrightarrow \{ 1,2,3,4 \}$, and we can say the two sets are size or cardinality 4.  We can also compare infinite sets using one-to-one correspondences.  The measure of the size of a set in this sense is called cardinality, and two sets have the same cardinality if they can be put into  one-to-one correspondence.

Let us check a few infinite sets that are of different sizes by measures other than cardinality (e.g., different length), and look to be very different at first glance.  Start with two intervals in $\mathbb R$ where one is 10 times longer.  Are the intervals $[0,1]$ and $[0,10]$ the same cardinality?  Yes, these are the same cardinality, and it is demonstrated by the one-to-one correspondence $x \leftrightarrow 10x$.

Next consider two intervals:  One of finite length and one of infinite length.  Are these intervals of different cardinality?  Let us take $[0,1)$ and $[0,\infty)$.  These sets are in fact the same cardinality.  A one-to-one correspondence to demonstrate this equality is  $x \leftrightarrow \tan\frac{\pi x}{2}$.

Now compare an interval and the natural numbers.   Do $\mathbb N$ and the interval $[0,1]$ have the same cardinality?  These cannot be put into a one-to-one correspondence.  To show that there does not exist any  one-to-one correspondence, we show that in any one-to-one correspondence between $\mathbb N$ and a subset of $S\subset [0,1]$, the subset is not all of $[0,1]$.  The argument is famous and is  known as \textit {Cantor's diagonalization argument}.

Suppose you have any correspondence between $\mathbb N$ and a subset of $[0,1]$.  We can show that not every number in $[0,1]$ is in the correspondence and therefore it is not a one-to-one correspondence between $\mathbb N$ and  $[0,1]$.

First note that every number in $[0,1]$ can be written as a decimal $0.d_1d_2d_3\cdots$,  e.g., $1= 0.999\cdots=0.\overline 9$, $\frac 12 = 0.5\overline 0$ or $0.4\overline 9$, $\frac 13 = 0.\overline 3$.  A correspondence between $\mathbb N$ and a subset of $[0,1]$ has the form

\begin{align*}
1 &\leftrightarrow 0.a_{11}a_{12}a_{13}a_{14}a_{15}\cdots a_{1\, n-1}a_{1\,n}a_{1\,n+1} \cdots\\
2 &\leftrightarrow 0.a_{21}a_{22}a_{23}a_{24}a_{25}\cdots a_{2\,n-1}a_{2\,n}a_{2\,n+1} \cdots\\
3 &\leftrightarrow 0.a_{31}a_{32}a_{33}a_{34}a_{35}\cdots a_{3\,n-1}a_{3\,n}a_{3\,n+1} \cdots\\
\vdots &\qquad \qquad\qquad\qquad\qquad \vdots\\
n &\leftrightarrow 0.a_{n1}a_{n2}a_{n3}a_{n4}a_{n5}\cdots a_{n\,n-1}a_{n\,n}a_{n\,n+1} \cdots\\
\vdots &\qquad \qquad\qquad\qquad\qquad \vdots
\end{align*}
We can now produce a number in $[0,1]$ that is not in the correspondence.  For $k\in\mathbb N$, let
\[d_k=\left\{
\aligned
3  &\quad\text {if $a_{k\, k} = 4$}\\
4   &\quad\text {if $a_{k\, k} \ne 4$}\\
\endaligned\right.\]
The number  $r=0.d_1d_2d_3d_4\cdots$ is not in the correspondence.  Note that $r$ has only one representation as an infinite decimal since the digits do involve 9's or 0's.  It is not the first number in the correspondence since $r$ and the first number have different first digits, i.e., $d_1 \ne a_{11}$.  Similarly $r$ is not the second number since $d_2 \ne a_{22}$.  This pattern holds in general.  The digits along the diagonal of the list distinguish the numbers in the list from $r$.  In other words  $r$ is not the $n^{th}$ number for any $n$ since $d_n \ne a_{nn}$.

The above argument shows that any list on the right hand side is not all $[0,1]$ and the set of numbers $[0,1]$ is not the same size as $\mathbb N$.  The cardinality of $\mathbb N$ is referred to as countably infinite and an infinite set (like $[0,1]$) that cannot be put into one-to-one correspondence with $\mathbb N$ is called uncountable.

What about the rational numbers?  

\section{Are the rational numbers countable or uncountable?}

One property the rational numbers possess is that between any two real numbers there are infinitely many rational numbers!  It is called being dense in the real line and the real numbers also possess this property.  When faced with the question of whether the rational numbers are countable like $\mathbb N$ or uncountable like $\mathbb R$, many people intuit that the rational numbers are uncountable because they are dense like the real numbers.  This is false.

The rational numbers are countable.  In fact, if a student successfully answers Problem~(\ref {thm:problem}), then they will have proved that there is a one-to-one correspondence between the positive rational numbers and $\mathbb N$.  The correspondence is simply:
\[ n \leftrightarrow a_n.\]
It is now easy to construct a one-to-one correspondence between the rational numbers and $\mathbb N$.  Let $s_1=0$, $s_{2n}=a_n$,  $s_{2n+1}=-a_n$ and the correspondence is:
\[ n \leftrightarrow s_n.\]

\section{Properties of the sequence}

Some properties of the sequence that will be required in subsequent sections are discussed.  First recall the recursion relations from Problem (\ref {thm:problem}):
\[ a_1 = 1, a_{2n} = a_n + 1, \text{ and } a_{2n+1} = \frac 1{a_{2n}}.\]
From the second equality we note that
$a_{2^m k}= a_k + m$  and in particular $a_{2^m}= m+1$.

Next we observe that the sequence in Problem~(\ref {thm:problem}) has the following properties:
\begin{properties}
\label{thm:properties}
\begin{align}
 a_1=1\\
a_{2k} = a_k +1 >1\\
a_{2k+1} = \frac 1{a_{2k}} <1 \text{ for } k>0
\end{align}
\end{properties}
The even indexed sequence elements are the positive rational numbers greater than 1 and the odd indexed rational numbers are the positive rational numbers less than or equal to 1.  In particular 
\[n \leftrightarrow a_{2n+1}\]
gives a one-to-one correspondence between the rational numbers in $(0,1)$ and $\mathbb N$, or in other words, the rational numbers in $(0,1)$ as the subsequence of odd indices starting at 3: $a_3, a_5, a_7, \cdots$.  

\section{A sequence of Johannes Kepler}
\label{sec:kepler}

The subsequence $a_3, a_5, a_7, \cdots$, which enumerates the rational numbers strictly between 0 and 1, was essentially given by Johannes Kepler in \textit{Harmonices Mundi, Book III}.  A translation is by Aiton, Duncan and Field (see \cite{kepler}) and is excerpted at the website \cite{PlanePath}.  While Kepler did not describe a sequence, he described the elements in an ordered manner by giving a tree structure.  His tree structure relates to the subsequence we defined as follows:

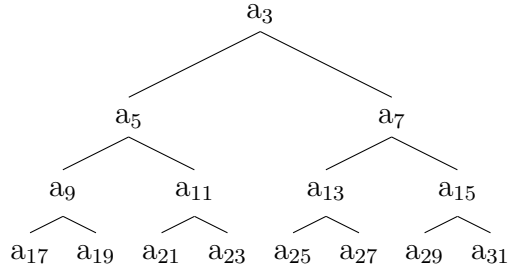
\Tree[.a_3 [.a_5 [.a_9 [.a_{17} ] [.a_{19} ]]
                  [.a_{11}  [.a_{21} ] [.a_{23} ] ]]
               [.a_7 [.a_{13}  [.a_{25} ] [.a_{27}  ] ]
                  [.a_{15} [.a_{29}  ] [.a_{31}  ]
                          ]]]

\vspace {3 mm}

Kepler's rules to propagate his tree are

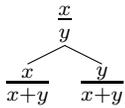
\vspace {3 mm}

\begin{enumerate}
\item[ ]
\hskip -0.3in
\Tree[.{$\frac xy$} [.{$\frac x{x+y}$} ]
               [.{$\frac y{x+y}$}  ]]
{\quad or in terms of the sequence the propagation is}
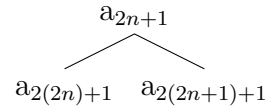
\Tree[.a_{2n+1} [.a_{2(2n) + 1}  ]
               [.a_{2(2n+1) + 1} ]]
\end{enumerate}
\vspace {2 mm}
It is a short exercise with the two recursion relations in Problem (\ref{thm:problem}) to show that if $\displaystyle a_{2n+1} = \frac xy$, then  $\displaystyle a_{2(2n) + 1} = \frac 1{\frac 1{a_{2n+1}}+1}= \frac x{x+y}$ and  $\displaystyle a_{2(2n+1) + 1} =\frac 1{a_{2n+1}+1}= \frac y{x+y}$.

Another interesting feature of Kepler's tree is that the subsequence obtained by moving down the right-hand edge of the tree are all ratios of successive Fibonacci numbers.  

\section{A solution to Problem (\ref {thm:problem}).}

The solution presented is very similar to Alexander Remorov's solution \cite{Remorov}.
We first show that the sequence includes all positive rational numbers.
Use induction on $k$ with the induction hypothesis:  if $\displaystyle\frac pq$ with $p,q\in\mathbb N$ has $p+q\le k$, then
$\displaystyle\frac pq$ occurs in the sequence.

The induction hypothesis holds for $k=2$ since $a_1=\frac 11 = 1$.
Now suppose $\displaystyle\frac rs$ is a rational number with $r+s=k+1$, and the induction hypothesis holds for $k$.  Note that $r\ne s$ by Properties (\ref{thm:properties}) 2 and 3.  Consider the case $r>s$.  By the induction hypothesis, $\displaystyle\frac {r-s}s =a_m$ for some $m$. Then
$\displaystyle\frac rs = 1+\frac {r-s}s =a_{2m}$.
Next, consider the case $s>r$.  By the first case, there is an $m$ with $a_m= r/s$.  Therefore,  by the recursion relation in Problem (\ref {thm:problem}), $a_{m+1}= s/r$.


We next show that there are no repetitions in the sequence so that the sequence has distinct terms.  Suppose the smallest index representing a number that occurs more than once is $m$ and $a_m = a_n$.  Note that $m>1$ by Properties (\ref{thm:properties}) 2 and 3.  Also note that by Properties (\ref{thm:properties}) 2 and 3, $m$ and $n$ are the same parity.  If $m$ and $n$ are even, then $a_{m/2} = a_{n/2}$, by the recursion in Problem (\ref {thm:problem}).  This contradicts $m$ being minimum.  Similarly, if $m$ and $n$ are odd, then $a_{m-1} = a_{n-1}$.  Again, it follows by the recursion in Problem (\ref {thm:problem}) that contradicts $m$ being minimum.

\section{Continued fraction representation of elements and base 2 representation of indices}

We can read the sequence entry for $a_n$ from the representation of $n$, the index of the sequence entry.  (A similar treatment is given in \cite {Czyz}).  First write $n$ in base two as
\[ n= 2^{m_0} + 2^{m_1}+2^{m_2} + \cdots+ 2^{m_k} \]
with $m_0$ a nonnegative integer, the other $m_i$'s positive integers, and  $m_i < m_{i+1}$.  Next define $n_0 = m_0$ and $n_i = m_i - m_{i-1}$ for $i=1,\cdots, k$.   We may then write
\begin{equation}
\label{eq:base2}
n= 2^{n_0} + 2^{n_0+n_1}+2^{n_0+n_1+n_2} + 2^{n_0+n_1+n_2+n_3} + \cdots + 2^{n_0+n_1+n_2+n_3+\cdots+n_k},
\end{equation}
which we may express as 
\begin{equation}
\label{eq:n}
n= 2^{n_0}(1+2^{n_1}(1+2^{n_2}(1+2^{n_3}(\cdots 2^{n_{k-1}}(1+ 2^{n_k})\cdots))))
\end{equation}
where the intergers $n_i$ satisfy $n_0\ge 0$ and $n_i> 0$ for $i=1,\cdots,k$.

To unravel the value of $a_n$, we can use the expression for $n$ in Eq (\ref {eq:n}) and the recursion relations in Problem (\ref{thm:problem}).  Start from the inside and work out:
\begin{align*}
a_{1+ 2^{n_k}} &= \frac 1{a_{2^{n_k}}} = \frac 1{ 1+ n_k}\\
a_{1+2^{n_{k-1}}(1+ 2^{n_k}) } &= \frac 1{n_{k-1} +\frac 1{1+ n_k}}\\
&\vdots \\
a_{ 2^{n_0}(1+2^{n_1}(1+2^{n_2}(1+2^{n_3}(\cdots 2^{n_{k-1}}(1+ 2^{n_k})\cdots)))) } &=n_0+\frac{1}{n_1 + \frac{1}{n_2+ \frac 1{\ddots\frac{1}{n_{k-1}+
\frac{1}{1+ n_k }}}}}
\end{align*}
This final expression for $a_n$ as a long compound fraction
\begin{equation}
\label{eq:continued}
a_n=n_0+\frac{1}{n_1 + \frac{1}{n_2+ \frac 1{\ddots\frac{1}{n_{k-1}+
\frac{1}{1+ n_k }}}}}
\end{equation}
is called a continued fraction.

The $n_i$ can be obtained from counting the number of zeros between the base two digits of $n$ as can be seen in Eq (\ref{eq:base2}).  Therefore, from the base two representation of the index $n$, one can read off the sequence value $a_n$ by counting the number of zeros between ones or the decimal point.  The instructions are as follows:

\begin{instructions}
\label{thm:instructions} \
\begin{enumerate}
\item  \label {itm:1} $n_0$ is the number of zeros between the decimal point and the first one (with perhaps $n_0$ being 0).
\item   \label {itm:2} $n_m$-1 is the  number of zeros between the $m$-th one and the next one for $0<m< k.$
\item  $n_k$-2 is the  number of zeros between the $k-1$-st one and the last ($k$-th) one.
\end{enumerate}
\end{instructions}

\noindent
Note the last nested denominator of the partial fraction is $1+n_k$.  For example, if $600= 1001011000.0$, so 
\[ a_{600} = 3+\frac{1}{1 + \frac{1}{2+\frac{1}{4}}}.\]
Furthermore for any $n>600$ that has a base two representation agreeing with the start of 600, i.e., $\cdots 1011000$, $a_n$ has the same beginning as a continued fraction as $a_{600}$ (up through $n_{k-1}=2$),
\[a_{600} = 3+\frac{1}{1 + \frac{1}{2+\frac{1}{\ddots}}}.\]

If you combine the observation that $a_n$ has the form given in Eq (\ref {eq:continued}) with Problem (\ref {thm:problem}), then you see that each positive rational number has a unique representation as a continued fraction of the form given in  Eq (\ref {eq:continued}), i.e., with for some whole number $k$, $n_0\ge 0$, and $n_i> 0$ for $i=1,\cdots,k$.\footnote{In some presentations, $n_k$ is allowed to be zero and then there may be two representations}.  It follows as a simple exercise that

\begin{theorem}
\label{thm:rational}
For every rational number $q$ there are unique choices of $z$ an integer, $k$ a whole number, and if $k>0$,  natural numbers $n_i$ for $i=1,\cdots,k$
so that $q$ may be uniquely represented as 
\[ z+\frac{1}{n_1 + \frac{1}{n_2+ \frac 1{\ddots\frac{1}{n_{k-1}+
\frac{1}{1+n_k }}}}}. \]
\end{theorem}

Finally, note that if you start with a positive rational number $q$ and write down its continued fraction representation, then you have a recipe for finding the place of $q$ in the sequence, i.e., for finding the $n$ with $q=a_n$.  The recipe for expressing $q$ as a continued fraction follows from an algorithm that is thought to go back to the Pythagoreans and is called the Euclidean Algorithm as it appeared in Euclid's \textit{Elements}.

\section{Extension to positive real continued fractions and a sequence indexed by the 2-adic integers.}

We move beyond Problem (\ref {thm:problem}) in this section to an extension of the sequence.  We still write of this extension as a sequence even though it is more properly a function.

We have already seen and proved that there is a one-to-one correspondence between the rational numbers and the continued fractions.  These continued fractions had finitely many quotients terminating with $n_k+1$.  We will call these finite continued fractions.  The notion of continued fractions may be extended to infinite expressions:

\begin{equation}
\label{eq:infinitecontinuedfraction}
z+\frac{1}{n_1 + \frac{1}{n_2+ \frac 1{n_3 +\frac 1{n_4+ \frac 1{\ddots }}}}}.
\end{equation}
for $z\in\mathbb Z$ and $n_i\in\mathbb N$.

\begin{theorem}
\label{thm:irrational}
There is a one-to-one correspondence between the infinite continued fractions, i.e., expression of the form given in  (\ref{eq:infinitecontinuedfraction}), and the irrational numbers.  The continued fraction (\ref{eq:infinitecontinuedfraction}) is nonnegative if and only if $z\ge 0$.
\end{theorem}

An infinite expression of the form (\ref{eq:infinitecontinuedfraction}) must be interpreted using the theory of limits just as infinite decimals must be interpreted using limits.  One may terminate the infinite continued fraction (\ref{eq:infinitecontinuedfraction}) at a finite level, say at $n_m$, to get the rational number $c_{m+1}$.  The $c_m$'s are called \textit {convergents} of the continued fraction.  The infinite continued fraction (\ref {eq:infinitecontinuedfraction}) is viewed as a limit of its sequence of convergents:
\begin{equation}
\label{eq:convergents}
z,
z+\frac{1}{n_1},\,
z+\frac{1}{n_1 + \frac{1}{n_2}},\,
z+\frac{1}{n_1 + \frac{1}{n_2+ \frac 1{n_3 }}},\,
z+\frac{1}{n_1 + \frac{1}{n_2+ \frac 1{n_3 +\frac 1{n_4}}}} \, \cdots.
\end{equation}
The limit $\lim_{m\to\infty} c_m$ always exists and is an irrational number, $r$.  The odd subsequence of convergents $\{c_{2n-1}\}_{n=1}^\infty$ monotonically increases to $r$, and the even subsequence of convergents $\{c_{2n}\}_{n=1}^\infty$ monotonically decreases to $r$.  The details are beyond this talk but reader may find them in many sources such as \cite {Olds}.

It is interesting to note that continued fraction representation of a number may be more easily accessed or understood than a decimal representation.  For example, while both $\sqrt 2$ and $e$ have infinite decimal representations that are not completely known, the continued fraction representations are known.  Square roots all have periodic continued fraction representations and

\[
\sqrt 2 = 1+\frac{1}{2 + \frac{1}{2+ \frac{1}{2+\frac{1}{2 +\frac{1}{ \quad \ddots}}}}}
\]
or $n_i =2$ for $i=1,2,3, \cdots$.  The representation of $e$ was first published in 1744 by Euler:

\[
e = 2+\frac{1}{1 + \frac{1}{2+ \frac{1}{1+\frac{1}{1 +\frac{1}{4+\frac{1}{1 +\frac{1}{\quad \ddots}}}}}}}
\]
or $n_0 =2$ and for $i>0$, $n_i=\left\{
\aligned
2k  &\quad\text {if $i=3k-1$}\\
1  &\quad\text {otherwise}\\
\endaligned\right.
$.
For a proof see \cite{Cohn}.

We now turn to the second aspect needed to extend the sequence.  The index set of the sequence used to obtain the rational numbers are the natural numbers.  Representing the index $n$ as a base two numeral allows us to directly produce continued fraction representation of $a_n$.  We extend the index set to a number system called the 2-adic integers to produce a type of 2-adic sequence that yields the positive real numbers without repetition. 

Infinite expressions like infinite continued fractions or infinite decimals are given meaning using the notion of limits.  For example the infinite decimal $0.333333333\cdots$ is given meaning as the limit of the sequence 0.3, 0.33, 0.333, 0.3333, etc.  The later terms are close together because their difference is small in the usual absolute value or norm.  The digits farther to the right of the decimal point are in smaller place values than those to the left in terms of the norm.

There are number systems introduced (in the late 1890's by Kurt Hensel) that extend the counting numbers using norms motivated by number theory.  For each prime number $p$, there is a number system called the $p$-adics.  These $p$-adic numbers turn out to be very important in mathematics.  For example, Wiles's proof of Fermat's last theorem uses p-adic numbers.  We consider only the 2-adics.  Every integer can be written as a power of two times an odd number, $n = 2^k m$ with $m$ odd.  The $2$-adic norm is $|n|_2 = |2^k m|_2 = 1/2^k$.  If we right counting numbers in base two, then the digits farther to the left are in the smaller place value in terms of the 2-adic norm.  The 2-adic integer numerals are base two numerals that extend infinitely to the left!  A 2-adic integer is of the form $\displaystyle z=\sum 2^{m_i}$ for $m_i$ an increasing finite or infinite sequence of counting numbers.  An example of is 
$\cdots 01010101010 = \sum_{m=1}^\infty 2^{2m-1} = -2/3$, which is a 2-adic integer.  (You can see it is -2/3 by dividing by 2, multiplying by 3, and adding 1 to get zero).

We can extend the index set of the sequence in Problem (\ref {thm:problem}) infinite 2-adic numerals by using part \ref {itm:1} and \ref {itm:2} of Instructions (\ref{thm:instructions}).  So for example, 
\[ a_{-2/3}=a_{\ \cdots 01010101010}=1+\frac{1}{2 + \frac{1}{2+ \frac{1}{2+\frac{1}{2 +\frac{1}{2+\frac{1}{2 +\frac{1}{\quad \ddots}}}}}}},\]
which happens to be $\sqrt 2$.  In general, following Instructions (\ref{thm:instructions}) means expressing a 2-adic integer as 
\begin{equation}
\displaystyle n=\sum_{m=0}^\infty 2^{\sum_{i=0}^{m} n_i}
\end{equation}
then the instructions yield 
\begin{equation}
a_{n} =n_0+\frac{1}{n_1 + \frac{1}{n_2+ \frac{1}{n_3 +\frac{1}{n_4 +\frac{1}{n_5 +\frac{1}{\ddots }}}}}}
\end{equation}

This sets up an one-to-one correspondence between the 2-adic integers and the positive real numbers by Theorems (\ref{thm:rational}) and (\ref{thm:irrational}).

We close by noting that although this map is not continuous, there is a consistency result.  Every 2-adic integer is a limit of counting numbers specified by the numeral representation since
\[
\lim_{m\to\infty} \sum_{i=0}^m 2^{m_i} = \sum_{i=0}^\infty  2^{m_i}.
\]

The consistency result is

\begin{theorem}
\label {thm:consistency}  Let $z=\sum_{k=0}^{\infty} 2^{\sum_{i=0}^{k} n_i}$ with $n_i\in\mathbb Z$, $n_0\ge 0$, and $n_i >0$ for $i>0$.  Then
\begin{equation} \label{eq:2-adic}  
\lim_{m\to\infty} a_{ \sum_{k=0}^{m} 2^{\sum_{i=0}^{k} n_i}} = n_0+\frac{1}{n_1 + \frac{1}{n_2+ \frac{1}{n_3 +\frac{1}{n_4 +\frac{1}{n_5 +\frac{1}{\ddots }}}}}} = a_{z}
\end{equation}
\end{theorem}

\begin{proof}
To show Eq.(\ref {eq:2-adic}), we need to refer to the theory of continued fractions.  The continued fraction in Eq.(\ref {eq:2-adic}) has a real limit $r$.  The $k$-th convergent is obtained by retaining truncating the continued fraction after the $k$-th denominator, i.e., in Eq.(\ref {eq:2-adic}) dropping, or setting to zero,$n_i$ for $i>k$.  The even convergents form an increasing sequence with limit  $a_z$ and the odd convergents form a decreasing sequence with limit $a_z$.  

Now,
$a_{ \sum_{k=0}^{m} 2^{\sum_{i=0}^{k} n_i}}$ is between its $(m-1)$-st and $(m-2)$-nd convergent, but the $(m-1)$-st and $(m-2)$-nd convergents of $a_{ \sum_{k=0}^{m} 2^{\sum_{i=0}^{k} n_i}}$ and $a_z$ are the same.  Eq.(\ref {eq:2-adic}) follows since $a_{ \sum_{k=0}^{m} 2^{\sum_{i=0}^{k} n_i}}$ is sandwiched between a monotonically increasing and decreasing sequences converging to $a_z$.
\end{proof}

\begin{example}  Consider right-hand edge of Kepler's tree in Section (\ref {sec:kepler}).  This edge consists  of ratios of successive Fibonacci numbers.  In terms of the sequence from Problem (\ref{thm:problem}), the edge is
\[ a_3, a_7, a_{15}, \cdots, a_{2^n -1}, \cdots.\]
If we write the indices in base two, it is
\begin{equation}
\label{eq:convergentsrecip}
a_{11}, a_{111}, a_{1111}, \cdots, a_{\underbrace{1\cdots 1}_{\text{$n$ one's}}}, \cdots.
\end{equation}
By Theorem (\ref {thm:consistency}),
\[\lim_{n\to\infty} a_{2^n -1} =   \lim_{n\to\infty} a_{1\cdots 1} = a_{\cdots 111}  \]
and 
\begin{equation}
\label{eq:reciprocal}
a_{\cdots 111} =  \frac{1}{1 + \frac{1}{1+ \frac{1}{1+\frac{1}{1 +\frac{1}{ \quad \ddots}}}}}
\end{equation}
by the instructions.  This continued fraction is the reciprocal of the golden ratio or $\displaystyle \frac {\sqrt 5 -1}2$.  The sequence of the right-hand edge of the tree, Eq. (\ref{eq:convergentsrecip}), is the sequence of convergents for the continued fraction (\ref{eq:reciprocal}).

\end{example}

\vspace {0.2 mm}

\noindent
\textbf {Acknowledgements}:  The author is grateful to Andrew McDaniel for reviewing this manuscript.

\end{document}